\documentclass[11pt]{article}
 \usepackage{amsmath, amssymb,amscd} %% matematika
 \usepackage{enumerate} %% seznami
 \usepackage[active]{srcltx} %% iskanje v izvorni datoteki
 \usepackage{amsthm}

\usepackage{natbib} %% literatura
\usepackage{amsthm} %% izreki...

\theoremstyle{definition}
\newtheorem{defn}{Definition}

\newtheorem{ex}{Example}

\theoremstyle{plain}
\newtheorem{thm}{Theorem}
\newtheorem{lemma}{Lemma}
\newtheorem{prop}{Proposition}
\newtheorem{cor}{Corollary}

\usepackage{mathptmx} % fonti roman

\newcommand{\NN}{\mathbb{N}}

\newcommand{\mset}{{\mathcal Q}}
\newcommand{\luint}[3]{{[#1\underline{#2}#3, #1\overline{#2}#3]}}

\usepackage{natbib} %% literatura
\author{Damjan \v{S}kulj}
\title{Interval matrix differential equations}

\begin{document}
\maketitle
\begin{abstract}
	The matrix differential equation $x'(t) = Q(t)x(t), x(0) = x_0$ is considered in the case where $Q(t)$ is an unspecified matrix function of time, with the only constraint that $Q(t)\in \mset$ for every $t$, where $\mset$ is a prescribed closed and convex set of matrices. We provide the solution of the generalised equation by defining an exponential of a set of matrices. Although the defintion is not directly applicable to calculate the solutions, we provide an approximate method along with the estimation of maximal possible error. In particular, the method allows estimating continuous time imprecise Markov chains. 
\end{abstract}

{\bf Key words } matrix differential equation, interval matrix, exponential of an interval matrix, continuous time imprecise Markov chain

{\bf Math. subj. class.} 15B15, 15A16

\section{Problem setting}\label{s-ps}
We consider the matrix differential equation of the form
\begin{equation}\label{mde}
 x'(t) = Qx(t),\qquad x(0) = x_0. 
\end{equation}
where $Q$ is replaced with a compact convex set of matrices, called an \emph{interval matrix}, usually we denote it by $\mset$; and the intial vector $x_0$ is replaced by an interval vector. Interval matrices have been proven to efficiently model the uncertainty in parameter estimates, such as in modelling discrete time Markov chains with unknown parameters (see for instance \citet{hart:98,decooman-2008-a,skulj:09}). One of our main motivations is the modelling of continuous time Markov chains with uncertain parameters. We build on the work of \citet{hart:ctmc}, who studied the related problem requiring that $Q(t)$ is a piecewise continuous matrix function whose values lie between given matrices $\underline Q$ and $\overline Q$. Moreover, he requires zero column sums of all matrices $Q(t)$ and that the elements of $x_0$ sum to one. He then shows that the sets of all possible solutions are compact and convex, and studies their properties when time $t$ approaches infinity. 

In this paper we relax most of the \citeauthor{hart:ctmc}'s assumptions. Thus we allow arbitrary compact convex sets of matrices and in addition we omit any continuity assumptions. With regard to solutions we provide a directly applicable methods for finding approximate solutions, based on linear programming techniques, together with error estimates. 

There are two most common interpretations of the equation \eqref{mde} if instead of a fixed matrix $Q$, a set of possible matrices $\mset$ is given. The first one is the sensitivity analysis interpretation, where it is assumed that the matrix $Q$ is constant in time although all that is known is that it is contained in $\mset$ (see for  instance \citet{oppen:88,goldsztejn:2009}).

In this paper however we adopt another interpretation, which is similar to the one adopted by \citet{hart:ctmc}. That is, we allow that the matrix $Q$ depends on time $t$ in some very general unspecified way. Unlike \citeauthor{hart:ctmc} we do not even make any continuity assumptions. The only requirement being that $Q(t)$ is an element of $\mset$ at every time $t$.

Another related problem is when the matrix $Q(t)$ is a known function of time $t$ (see e.g. \citet{bohner:2001,dacunha:2005} and references therein). Our approach would rougly correspond to finding the bounds for all possible such solutions when only the bounds are given for $Q(t)$.

When the general equation \eqref{mde} is considered the closed form solution is obtained using matrix exponentials:
\begin{equation*}\label{general-solution}
 x(t) = e^{tQ}x_0. 
\end{equation*}
The approach to the solution of the generalised problem that we propose is based on the generalisation of the matrix exponential for the case of interval matrices. We show that this is indeed possible and provide methods to find approximations, with arbitrary precision that can be estimated in advance.

The paper has the following structure. In Section~\ref{s-imde} we formalize the generalisation of matrix differential equations. The solution of the genaralised equations is based on the defintion of an exponential of interval matrix defined and analysed in Section~\ref{s-eim}, where a method for practical approximations is presented along with the estimation of errors. In Section~\ref{s-simde} we prove that the exponentials of interval matrices indeed provide the solution of the interval matrix differential equation. Finally, in Section~\ref{s-rcimc} we explain how the method can be applied to continuous time imprecise Markov chains, and give an example. Some of the technical auxiliary results are listed in the appendix.

%-----
%\begin{enumerate}
% \item Exactly formulate the equations \eqref{initial_conds}, \eqref{diff_ineq_inf} and \eqref{diff_ineq_sup} for the case where $\mset$ is an \emph{interval matrix with separately specified rows} (see the next section).
% \item Show that the exact solutions $x(t)$ to the above equations are expressible in terms of intervals of vectors (also defined in the next section). 
% \item Define an \emph{exponential of an interval matrix} so that replacing it in the equation \eqref{general-solution} will provide the general solution to the interval differential equation. 
% \item Provide efficient methods for calculation. 
%\end{enumerate}

\section{Interval matrix differential equations}\label{s-imde}
\subsection{Intervals of vectors and matrices}\label{ss-ivm}
All vectors and matrices used here are finite-dimensional and all vectors are assumed to be column vectors. We will write $x\le y$ when $x_i\le y_i$ holds for all their components. Moreover, we will involve sets of vectors and sets of matrices. All operations between sets of matrices or sets of vectors are meant to be elementwise. Thus, for instance, given a set of matrices $\mset$ and a set of vectors $\mathcal X$, we denote 
\begin{equation*}
 \mset \mathcal X = \{ Qx\colon Q\in \mset, x\in \mathcal X \}. 
\end{equation*}
\begin{defn}
Let $\underline x$ and $\overline x$ be vectors such that $\underline x\le \overline x$. Then the set of  vectors $[\underline x, \overline x]:=\{ x\colon \underline x\le x \le \overline x\}$ is called an {\em interval vector}. 
\end{defn}
\begin{defn}
 A set of matrices $\mset$ has {\em separately specified rows} if for every $Q, Q'\in \mset$ the matrix $\tilde Q$, such that $\tilde Q_{ij} = Q_{ij}$ for all $j$ and $i\neq i_0$ and $\tilde Q_{i_0j} = Q'_{i_0j}$ for all $j$, also belongs to $\mset$. 
\end{defn}
This means that rows of matrices in $\mset$ can be chosen independently from some sets of row vectors. In the case, such as \citeauthor{hart:ctmc}'s, where the set of matrices consists of all matrices $Q$ that satisfy the equation $\underline Q\le Q\le \overline Q$, this property is already implicity assumed by the defintion. 
\begin{prop}\label{prop_interval}
Let $\mset$ be a compact convex set of matrices with separately specified rows and $x$ a vector. Then there exist matrices $Q_0, Q_1\in\mset$ such that $Q_0x \le Qx$  and $Q_1x\ge Q x$ for every $Q\in\mset$.
\end{prop}
\begin{proof}
 Let $\underline y_i = \min_{Q\in \mset} (Qx)_i$ and let $Q_i$ be a matrix that minimizes $(Qx)_i$. Then let $\tilde Q$ be the matrix whose $i$th row is equal to the $i$th row of $Q_i$ for every $i$. Then, clearly, $(\tilde Qx)_i = (Q_ix)_i$ which implies that $(\tilde Qx)_i\le (Qx)_i$ for every $Q\in \mset$ and every component $i$, and this completes the proof. 
\end{proof}
If $\mset$ is a compact convex set of matrices with separately specified rows and $x$ a vector then their product is the interval
\begin{equation}\label{Q-interval}
 \mset x = [\underline Qx, \overline Qx], 
\end{equation} 
where $\underline Qx := \min_{Q\in \mset} Qx$ and $\overline Qx := \max_{Q\in \mset} Qx$, which exist by the above proposition. Moreover, when an interval matrix is multiplied with an interval vector, the result is again an interval:
\begin{equation*}
\mset [\underline x, \overline x]= [\underline Q\,\underline x, \overline Q\,\overline x].
\end{equation*}
Thus every compact convex set of matrices $\mset$ with separately specified rows defines a (non-linear) transformation of the set of interval vectors to itself. 
\begin{defn}
An {\em interval matrix} is a compact and convex set of matrices $\mset$ with separately specified rows. 
\end{defn}
\subsection{The formulation of the interval matrix differential equation}\label{ss-fimde}
We have shown that an interval matrix $\mset$ maps an interval vector $\luint{}{x}{}$ to the interval vector $[\underline Q\,\underline x, \overline Q\,\overline x]$. Consider all vector valued maps $x(t)$ satisfying the following inequalities:
\begin{align}
 & x(0) \in \luint{}{x}{_0} \label{initial_conds} \\
 & \limsup_{h\to 0} \frac{x(t+h)-x(t)}{h} \le \overline{Q}x(t) \label{diff_ineq_sup}\\
%  \intertext{and}
 & \liminf_{h\to 0} \frac{x(t+h)-x(t)}{h} \ge \underline{Q}x(t), \label{diff_ineq_inf}
\end{align}
for every $t\ge 0$. 

Let $\mathcal X$ denote the set of all possible solutions of the above system of differential inequalities, i.e. the set of all vector functions $x(t)$ satisfying the inequalities \eqref{initial_conds}, \eqref{diff_ineq_sup} and \eqref{diff_ineq_inf}. We are concerned with the sets of vectors $\mathcal X(t) := \{ x(t) \colon x\in \mathcal X\}$ that the possible solutions can take at time $t$. We will show that every set $\mathcal X(t)$ is an interval vector of the form $[\underline x(t), \overline x(t)]$. Moreover, we will provide the method to efficiently calculate its bounds. 

The solution of the classical matrix differential equation \eqref{mde} is obtained by multiplying the initial vector $x_0$ with the exponential of the matrix $Q$ to obtain $x(t) = e^{tQ} x_0$. In the next section we define a generalised exponential of an interval matrix which will provide the solution of the generalised equation. 
\subsection{Computational considerations}\label{ss-cc}
Two crucial considerations when doing computation with interval matrices are needed. One is the representation of interval matrices and the other is the computation of the interval  bounds \eqref{Q-interval}. 

Convex sets, and in particular iterval matrices, are represented via sets of constraints that in practical cases are usually finite and allow the computation of the bounds of the intervals of the form \eqref{Q-interval} using linear programming techniques. One such example of constraints is the most commonly used interval of matrices
\begin{equation*}\label{matrix-interval}
 [\underline Q, \overline Q] = \{ Q\colon \underline Q \le Q\le \overline Q\}.
\end{equation*}
Problems usually occur when the set of matrices is transformed in a non linear way. In our case the transformation is the exponential of an interval matrix defined in the next section. When such transformations are applied, not only it is not easy to find the constraints representing the transformed set, but also the cardinality of the set of such constraints may exponentially increase or even become infinite. For this reason we strive to besides giving the adequate theoretical definitions we also develop methods for efficient calculations, which are usually approximate methods. 

\section{The exponential of an interval matrix}\label{s-eim}
\subsection{Definition}\label{ss-d}
Our aim is defining a set of matrices $\exp(\mset)$ that will provide the solution to the set of inequalities \eqref{initial_conds}, \eqref{diff_ineq_sup} and \eqref{diff_ineq_inf} in the following sense:
\begin{equation*}
 \mathcal X(t) = \exp(t\mset)[\underline x_0, \overline x_0] = \{ Px \colon P\in \exp(t\mset), x \in [\underline x_0, \overline x_0] \}.
\end{equation*}
We have explained in Section~\ref{s-ps} that there are different interpretations of interval matrix differential equations. When sensitivity analysis interpretation is used, i.e. when it is assumed that the differential equation in question is a classical matrix differential equation with an unknown matrix $Q$, that is known to belong to the set $\mset$, the adequate way to calculate its exponential is to take the set of all exponentials of matrices in $\mset$:
\begin{equation}\label{pseudo_exponential}
 \exp'(\mset) =  \{ e^{Q}\colon Q\in\mset \}.
\end{equation}
This type of exponential and the methods for its efficient calculations have been provided by \citet{goldsztejn:2009,oppen:88}. 

The above exponential, however, is not suitable when our interpretation is adopted. Let us explain the main reason, why the exponential \eqref{pseudo_exponential} does not solve the set of inequalities \eqref{initial_conds}, \eqref{diff_ineq_sup} and \eqref{diff_ineq_inf}. Suppose we are trying to find the upper bound $\overline x(\Delta t_1)$ for the set of solutions at time $\Delta t_1$. We may have, for instance:
\begin{equation*}\label{lin_approx}
 \overline x(\Delta t_1) = e^{\Delta t_1Q_1}\overline x_0.
\end{equation*}
where $Q_1$ is the matrix in $\mset$ that maximizes the value $e^{\Delta t_1Q}\overline x_0$. At time $\Delta t_1+\Delta t_2$ let $Q_2$ be the matrix that maximizes $e^{\Delta t_2Q}\overline x(\Delta t_1)$. But the resulting upper vector $\overline x(\Delta t_1+\Delta t_2)$ is then in general greater than $e^{(\Delta t_1+\Delta t_2)Q}\overline x_0$ for every $Q\in \mset$. 

The exponential serving our purposes is instead defined as follows. 
\begin{defn}\label{interval_exponential}
 Let $\mset$ be an interval matrix with separately specified rows. Then we define
 \begin{equation}\label{interval_exponential_formula}
  e^{\mset} = \mathrm{cl}\left\{ \prod_{i=1}^n e^{d_iQ_i}\colon n\in\NN, \sum_{i=1}^nd_i = 1,  Q_i\in \mset, \text{ for } 1\le i \le n\right\},
 \end{equation}
 where $\mathrm{cl}$ denotes the closure of the set. 
\end{defn} 
Since interval matrices are compact by definition, this implies that the norms of their elements are uniformly bounded by a constant, say $M$. By Lemma~\ref{norm_exp} it easily follows that the norms of elements in $e^{\mset}$ are uniformly bounded as well with the constant $e^M$. Hence, $e^{\mset}$ is a closed set of matrices with uniformly bounded norms, which implies that it also compact. Every element of $P\in e^{\mset}$ is therefore the limit of a sequence $\{ P_n \}$ of the form 
\begin{equation}\label{sequence-exponential}
 P_n = \prod_{i=1}^{k_n} e^{d^n_iQ^n_i},
\end{equation}
where $\sum_{i = 1}^{k_n}d^n_i = 1, d^n_i\ge 0, Q^n_i\in \mset$ for every $n\in \NN$ and $1\le i\le k_n$. 
\begin{lemma}\label{additivity-exponential}
  For every $\alpha$ and $\beta$ we have that
  \begin{equation*}
   e^{(\alpha+\beta) \mset } = e^{\alpha \mset }e^{\beta \mset }.
  \end{equation*}
\end{lemma}
\begin{proof}
 Clearly $e^{\alpha \mset }e^{\beta \mset }\subseteq e^{(\alpha+\beta) \mset }$. To see the converse inclusion take an element of the form $\prod_{i=1}^n e^{d_iQ_i}$ where $\sum_{i=1}^n d_i = \alpha + \beta$ and let $k$ be the largest index such that $\sum_{i=1}^{k-1} d_i \le \alpha$. Let $d_k = d'_k+d''_k$ such that $\sum_{i=1}^{k-1} d_i + d'_k = \alpha$. Then we have that 
 \begin{equation*}
  \prod_{i=1}^n e^{d_iQ_i} = \left(\prod_{i=1}^{k-1} e^{d_iQ_i} e^{d'_kQ_k}\right) \left(e^{d''_kQ_k} \prod_{i=k+1}^{n} e^{d_iQ_i}\right),
 \end{equation*}
which is an element of $e^{\alpha \mset }e^{\beta \mset }$. 

Let $P$ be an arbitrary element of $e^{(\alpha+\beta) \mset}$. Then there is a sequence $\{ P_n\}$ of the form \eqref{sequence-exponential} such that $P = \lim_{n\to\infty} P_n$. Now, by the above considerations, every $P_n$ is a product of $P'_n\in e^{\alpha \mset }$ and $P''_n\in e^{\beta \mset }$, and since both $e^{\alpha \mset }$ and $e^{\beta \mset }$ are compact sets, both sequences $\{ P'_n\}$ and $\{ P''_n\}$ have subsequences that converge to some $P'$ and $P''$ respectively, whence $P = P'P''$ immediately follows and completes the proof. 
\end{proof}
For simplicity of notation we will denote an element of $e^{\mset}$ by $\tilde e^Q$, although, as follows from the definition there might be no $Q\in \mset$ such that $\tilde e^Q = e^Q$. 
\subsection{Partitions}\label{ss-p}
Let $\mathcal T = \{ t_0, t_1, \ldots, t_n \colon t_{i-1}< t_i, \text{for every } 1\le i \le n\}$ be a {\em partition} of the interval $[t_0, t_n]$. Suppose that $\mathcal T'= \{ t'_0, t'_1, \ldots, t'_m\}$ is another partition such that $t'_0 = t_0$ and $t'_m = t_n$ and $\mathcal T'\subseteq \mathcal T$. Then we will say that $\mathcal T'$ is a {\em refinement} of $\mathcal T$.

For every partition $\mathcal T$ define 
\begin{equation}\label{exp-partition}
 e_\mathcal T^\mset = \left\{ \prod_{i=1}^n e^{(t_i-t_{i-1})Q_i}\colon Q_i\in \mset, \text{ for } 1\le i \le n\right\},
\end{equation}
which is a closed set if $\mset$ is closed. 
Clearly we have that $e_\mathcal T^\mset \subseteq e_{\mathcal T'}^\mset$ if $\mathcal T'$ is a refinement of $\mathcal T$. The definition \eqref{interval_exponential_formula} can now be equivalently rewritten as
\begin{equation*}
 e^{\mset} = \mathrm{cl}\bigcup_{\mathcal T} e_\mathcal T^\mset, 
\end{equation*}
where $\mathcal T$ runs over all possible partitions of the unit interval $[0, 1]$.

Further, for every partition $\mathcal T= \{ t_0, t_1, \ldots, t_n\}$, we denote
\begin{align*}
 \delta (\mathcal T) & = \min_{1\le i\le n} t_i - t_{i-1}, \\
 \Delta (\mathcal T) & = \max_{1\le i\le n} t_i - t_{i-1} \text{ and }\\
 | \mathcal{T}| & = n.
\end{align*}
We will call a refinement $\mathcal T'$ of $\mathcal T$ an {\em elementary refinement} if there is at most one $t'_k\in \mathcal T'$ in every interval $[t_{i-1}, t_i]$ for $1\le i \le n$. Further, let $\mathcal T$ be a partition. Then we recursively define the following sequence of refinements:
\begin{equation*}
 \mathcal T^{(1)} = \mathcal T \cup \left\{ \frac{t_{i-1}+t_i}{2}\colon t_i \in \mathcal T, 1\le i \le n \right\} 
\end{equation*}
and 
\begin{equation*}
 \mathcal T^{(n)} = (\mathcal T^{(n-1)})^{(1)}.
\end{equation*}
Every partition in this sequence is therefore an elementary refinement of the previous one.
\begin{lemma}\label{lemma-subdivs}
 Let $\mathcal T$ and $\mathcal T'$ be partitions such that $\Delta (\mathcal T')< \delta (\mathcal T)$. Then $\mathcal T\cup \mathcal T'$ is an elementary refinement of $\mathcal T'$. 
\end{lemma}
\begin{proof}
 Let $t_k\in [t'_{i-1}, t'_i]$ where $t'_{i-1}$ and $t'_i$ are two consecutive elements of $\mathcal T'$. Then $\Delta(\mathcal T') < \delta(\mathcal T)\le t_{k+1}-t_k$ and therefore $t_{k+1}$ cannot lie in the same interval. 
\end{proof}
\subsection{Approximations}\label{ss-a}
The calculations involving the interval matrix exponentials cannot be performed directly, as is the case with the exponential in the case of a single matrices, but rather using approximations defined in the previous subsection. In this subsection we will derive an approximation method based on linear programming techniques along with error estimation. The accuracy of the approximations of the form \eqref{exp-partition}, which also are sets of matrices, will be measured using the \emph{Hausdorff metric} induced by the operator norm for matrices. Let $x$ be an element in a normed space and $\mathcal Y$ a compact set. Then the Hausdorff distance between them is 
\begin{equation*}\label{hausdorff_1}
 d_H\left( x,  \mathcal Y\right) = \min_{y\in \mathcal Y} \left\| x-y \right\|
\end{equation*}
and the distance between two compact sets is 
\begin{equation*}\label{hausdorff}
 d_H\left( \mathcal X, \mathcal Y\right) = \max \left\{ \max_{x\in \mathcal X} d_H\left(x,  \mathcal Y\right), \max_{y\in \mathcal Y}d_H\left( y, \mathcal X \right) \right\}.
\end{equation*}
It follows from the above definition that if $\mathcal X\subseteq \mathcal Y$ then
\begin{equation*}
d_H(\mathcal X,  \mathcal Y) = \max_{y\in \mathcal Y}d_H\left( y, \mathcal X \right).  
\end{equation*}
For every set $\mathcal X$ let $\| \mathcal X\| = \sup_{x\in \mathcal X}\| x\|$. 
\begin{defn}
Given a partition $\mathcal T$ with $|\mathcal T| = N$ and $\Delta(\mathcal T)=D$, and an interval matrix $\mset$ with $\|\mathcal  Q\| = M$, let
\begin{equation*}
 A(\mset, \mathcal T) = N(DM)^2e^{(1+D)M} (1.5+3e^{DM}).
\end{equation*} 
\end{defn}
\begin{prop}\label{A-Tn}
\begin{equation*}
 A(\mset, \mathcal T^{(n)})\le 2^{-n} A(\mset, \mathcal T).
\end{equation*} 
\end{prop}
\begin{proof}
 We have that $|\mathcal T^{(n)}| = 2^nN, \Delta(\mathcal T^{(n)})=2^{-n}D$ and $e^{(1+\frac{D}{2^n})M} (1.5+3e^{\frac{D}{2^n}M})\le e^{(1+D)M} (1.5+3e^{DM})$, whence the above inequality easily follows.
\end{proof}
\begin{lemma}\label{approx-lemma}
 Let $\mathcal T'$ be an elementary refinement of $\mathcal T$. Then  
 \begin{equation}\label{approx-lemma-equation}
  d_H\left(e_{\mathcal T}^\mset, e_{\mathcal T'}^\mset\right) \le A(\mset, \mathcal T).
 \end{equation}
\end{lemma}
\begin{proof}
 We will assume that every interval $[t_{i-1}, t_i]$ contains one $t'_k\in\mathcal T'$, since otherwise, we can form a refinement of $\mathcal T'$ with the centers added of those intervals $[t_{i-1}, t_i]$ that do not contain elements from $\mathcal T'$. The result would be a finer partition $\tilde{\mathcal T}'$, and such that $d_H\left(e_{\mathcal T}^\mset, e_{\mathcal T'}^\mset\right)\le d_H\left(e_{\mathcal T}^\mset, e_{\tilde{\mathcal T}'}^\mset\right)$. 

 Now denote, for every $1\le i\le n, d_i=t_i-t_{i-1}, d'_i = t_i - t'_k$ and $d''_i = t'_k - t_{i-1}$, where $t'_k$ is the unique element of $\mathcal T'$ contained in the interval $[t_{i-1}, t_i]$.  Further, let $D=\Delta(\mathcal T)$, and then we have that $d'_i$ and $d''_i\le D$ for every $1\le i\le n$. 

 Because of $e_{\mathcal T}^\mset\subseteq  e_{\mathcal T'}^\mset$ we have that
 \begin{equation*}
  d_H\left(e_{\mathcal T}^\mset, e_{\mathcal T'}^\mset\right) = \max_{\tilde e^{Q'}\in e_{\mathcal T'}^\mset}\min_{\tilde e^{Q}\in e_{\mathcal T}^\mset}\left\|\tilde e^{Q'} - \tilde e^{Q} \right\|.
 \end{equation*}
 Every element of $e_{\mathcal T'}^\mset$ is of the form 
 \begin{equation}\label{elt_of_refined}
  \tilde e^{Q'} = \prod_{i=1}^n e^{d'_iQ'_i}e^{d''_iQ''_i}.
 \end{equation}
 Convexity of the set $\mset$ implies that the convex combinations $Q_i = \frac{d'_i}{d_i}Q'_i +\frac{d''_i}{d_i}Q''_i$ of the elements of $\mset$, for every $1\le i\le n$, belong to $\mset$ as well. To every element of the form \eqref{elt_of_refined} corresponds the element
 \begin{equation*}\label{elt_of_partit}
  \tilde e^{Q} = \prod_{i=1}^n e^{d_iQ_i}.
 \end{equation*}
 Using Lemma~\ref{refining}, we obtain that
 \begin{align*}
 \left\|\tilde e^{Q'}  -\tilde e^Q \right\| = 
 \left\| \prod_{i=1}^n e^{d'_iQ'_i}e^{d''_iQ''_i} - \prod_{i=1}^n e^{d_iQ_i}\right\|  \le n(DM)^2e^{(D+1)M}(1.5+3e^{DM}),
 \end{align*}
 which implies \eqref{approx-lemma-equation} and thus completes the proof. 
\end{proof}
\begin{cor}\label{dist-t-tn}
 Let a partition $\mathcal T$ and an interval matrix $\mset$ be given. Then 
 \begin{equation*}
  d_H(e^\mset_{\mathcal T}, e^\mset_{\mathcal T^{(n)}}) \le 2A(\mset, \mathcal T)
 \end{equation*}
 for every $n>0$. 
\end{cor}
\begin{proof}
 By Proposition~\ref{A-Tn} and Lemma~\ref{approx-lemma}, we obtain
 \begin{multline*}
  d_H(e^\mset_{\mathcal T}, e^\mset_{\mathcal T^{(n)}})  \le \sum_{k=0}^{n-1} d_H(e^\mset_{\mathcal T^{(k)}}, e^\mset_{\mathcal T^{(k+1)}}) \\ \le \sum_{k=0}^{n-1} A(\mset, \mathcal T^{(k)}) 
   \le \sum_{k=0}^{n-1} \frac{A(\mset, \mathcal T)}{2^k} 
   \le 2A(\mset, \mathcal T), 
 \end{multline*}
 where $\mathcal T^{(0)}:=\mathcal T$.
\end{proof}
\begin{thm}\label{thm_approx}
 Let a partition $\mathcal T$ and an interval matrix $\mset$ be given. Then we have that
 \begin{equation}\label{approx-final}
  d_H(e^\mset_{\mathcal T}, e^\mset) \le 2A(\mset, \mathcal T).
 \end{equation}
\end{thm}
\begin{proof}
 Denote $N=|\mathcal T|, M=\|\mathcal  Q\|$ and $D=\Delta(\mathcal T)$. Let $\varepsilon>0$ be arbitrary and $\mathcal T_1$ a partition such that $d_H(e^\mset_{\mathcal T_1}, e^\mset)\le \varepsilon$. Further let $n$ be large enough that $A(\mset, \mathcal T^{(n)})\le \varepsilon$ and that $\Delta(\mathcal T^{(n)})= \frac{D}{2^n}< \delta(\mathcal T_1)$. Then let $\tilde{\mathcal T}=\mathcal T^{(n)}\cup \mathcal T_1$. We have that $\Delta(\tilde{\mathcal T}) < \frac{D}{2^n}$ and, by Lemma~\ref{lemma-subdivs}, $\tilde{\mathcal T}$ is an elementary refinement of $\mathcal T^{(n)}$ and also a refinement of $\mathcal T_1$. Therefore, $d_H(e^\mset_{\tilde{\mathcal T}}, e^\mset) \le d_H(e^\mset_{\mathcal T_1}, e^\mset)\le \varepsilon$ and $d_H(e^\mset_{\tilde{\mathcal T}}, e^\mset_{\mathcal T^{(n)}}) \le A(\mset, \mathcal T^{(n)}) \le \varepsilon$, by Lemma~\ref{approx-lemma}, which implies that $d_H(e^\mset, e^\mset_{\mathcal T^{(n)}})\le 2\varepsilon$ and then, by Corollary~\ref{dist-t-tn}, $d_H(e^\mset, e^\mset_{\mathcal T})\le 2(A(\mset, \mathcal T)+\varepsilon)$, and since $\varepsilon$ can be made arbitrarily small, this implies \eqref{approx-final}.
\end{proof}
\begin{cor}
 Let $\mathcal T$ be a partition with $\Delta(\mathcal T) = D, \delta (\mathcal T)=d$ and $\|\mset \| = M$. Then 
 \begin{equation*}
 d_H(e^\mset_{\mathcal T}, e^\mset) \le \frac{2D^2}{d}M^2e^{(1+D)M} (1.5+3e^{DM}).
 \end{equation*} 
\end{cor}
\begin{proof}
 The inequality is a direct consequence of Theorem~\ref{thm_approx} and the fact that $|\mathcal T| \le \frac 1d$.
\end{proof}

\begin{cor}
 Let $\mathcal T$ be a partition with $\Delta(\mathcal T) = \delta (\mathcal T)=D$ and $\|\mset \| = M$. Then 
 \begin{equation*}
 d_H(e^\mset_{\mathcal T}, e^\mset) \le 2DM^2e^{(1+D)M} (1.5+3e^{DM}).
 \end{equation*} 
\end{cor}
For actual calculations, however, calculating exponential at every time slice is not very practical. We now show that the matrix exponential can be approximated with linear factors. 
\begin{thm}\label{thm_linear}
 Let $\mathcal T$ be a partition of the unit interval and $\mset$ an interval matrix. Further let $\Delta(\mathcal T)=D, |\mathcal T| = N$ and $\| \mset\| = M$. Denote $d_i = t_i-t_{i-1}$ for every $1\le i \le N$. Then for every $\tilde e^Q\in e^{\mset}$ there exist matrices $Q_1, \ldots, Q_N$ so that
 \begin{equation}\label{thm_linear_eq}
  \left\| \tilde e^{Q} - \prod_{i=1}^N (I+d_iQ_i) \right\| \le N(DM)^2(2e^{(1+D)M}(1.5+3e^{DM})+0.5e^{M}). 
 \end{equation}
\end{thm}
\begin{proof}
 By Theorem~\ref{thm_approx} there exists $\tilde e^{Q'} = \prod_{i=1}^N e^{d_iQ_i}\in e^\mset_{\mathcal T}$ such that $d\left( \tilde e^{Q'}, \tilde e^{Q} \right) \le 2A(\mset, \mathcal T).$ By Corollary~\ref{approx-linear-alt} we have that 
 \begin{equation*}
  \left\| \prod_{i=1}^N e^{d_iQ_i}-\prod_{i=1}^N (I+d_iQ_i) \right\| \le \frac{N}{2}(DM)^2e^M.
 \end{equation*}
 Combining the above inequalities gives \eqref{thm_linear_eq}.  
\end{proof}
\begin{cor}
 Let $\mset$ be an interval matrix with $\|\mset \| = M$. Then for every $\tilde e^Q\in e^{\mset}$ and every $N\in \NN$ there exist matrices $Q_1, \ldots, Q_N$ so that
 \begin{equation*}
  \left\| \tilde e^{Q} - \prod_{i=1}^N \left(I+\frac 1N Q_i\right) \right\| \le 2\frac{M^2}{N}e^{\left(1+\frac 1N\right)M} (1.5+3e^{\frac MN}+0.5e^M).
 \end{equation*} 
\end{cor}

\section{Solutions of the interval matrix differential equation}\label{s-simde}
In this section we prove that the exponential of an interval matrix $\mset$ indeed provides the solution to the inequalities \eqref{initial_conds}, \eqref{diff_ineq_sup} and \eqref{diff_ineq_inf}, and that the solutions are interval vectors. 
\begin{thm}\label{interval-separately}
  Let $\mset$ be an interval matrix. Then, for every $x$, the set $e^\mset x$ is an interval, which we denote by $[\underline{e^\mset}x, \overline{e^\mset}x]$. 
\end{thm}
\begin{proof}
 We will use the same notation as in the proof of Theorem~\ref{thm_linear}. 
 
 Let $\mathcal T$ be a partition of the unit interval such that the right hand side in \eqref{thm_linear_eq} is smaller than $\varepsilon$, where $\varepsilon>0$ is arbitrary. Because the rows of $\mset$ are separately specified, so are the rows of $I+d_i\mset$ for every $1\le i\le N$ and therefore, by Proposition~\ref{prop_interval} and the discussion following it, the set $\{\prod_{i=1}^N (I+d_iQ_i)x\colon Q_i\in\mset \text{ for every } 1\le i \le N \}$ is an interval $[\underline x_\mathcal T, \overline x_\mathcal T]$. Moreover, for every $\tilde e^Q\in e^\mset$ we have that $\overline x_\mathcal T \ge \tilde e^Qx - \varepsilon \| x\|$ and similarly $\underline x_\mathcal T \le \tilde e^Qx + \varepsilon \| x\|$. 
 
 Although $\underline x_\mathcal T$ and $\overline x_\mathcal T$ may not belong to $e^\mset x$, there are, by Theorem~\ref{thm_linear}, some $\underline y_\mathcal T$ and $\overline y_\mathcal T\in e^\mset x$ such that $\| \underline x_\mathcal T- \underline y_\mathcal T\|$ and $ \| \overline x_\mathcal T- \overline y_\mathcal T\|\le \varepsilon$. This together with closedness of $e^\mset x$ clearly implies that $\underline x = \inf_{\mathcal T} \underline x_\mathcal T$ and $\overline x = \sup_{\mathcal T} \overline x_\mathcal T$ belong to $e^\mset x$ which is then equal to the interval $[\underline x, \overline x]$.
\end{proof}
\begin{cor}
 Let $\mset$ be an interval matrix. Then, for every interval vector $[\underline x, \overline x]$, the set $e^\mset [\underline x, \overline x]$ is an interval, which we denote by $[\underline{e^\mset}\, \underline x, \overline{e^\mset}\, \overline x]$. 
\end{cor}
Every $\tilde e^Qx_0$, where $\tilde e^Q\in e^\mset$ and $x_0\in [\underline x_0,  \overline x_0]$, satisfies the inequalities \eqref{initial_conds}, \eqref{diff_ineq_sup} and \eqref{diff_ineq_inf} by construction. The following theorem shows that the converse also holds. 
\begin{thm}
 Let $x(t)$ satisfy \eqref{initial_conds},  \eqref{diff_ineq_sup} and \eqref{diff_ineq_inf}, where $\mset$ is an interval matrix. Then $x(1)\in [\underline{e^\mset}x_0, \overline{e^\mset}x_0]$.
\end{thm}
\begin{proof}
 Because of symmetry we only prove that $x(1)\le \overline{e^\mset}x_0$. Suppose that $\| \mset\| \le M$ and let $\varepsilon>0$ be given. Compactness of the interval $[0, 1]$ implies the existence of $n$ large enough that 
 \begin{equation*}
  x\left(t+\frac 1n\right) - x(t) \le \frac 1n \overline{\mset}x(t) + \varepsilon \frac 1n
 \end{equation*}
for every $t\in [0, 1-1/n]$. Further, let $n$ be large enough so that 
\begin{equation}\label{dist_lin_exp}
 \left\|\prod_{i=1}^n \left(I + \frac 1n Q_i\right) - \prod_{i=1}^n \left(e^{\frac 1n Q_i}\right) \right\| \le \varepsilon,  
\end{equation}
for every sequence of matrices $Q_1, \ldots, Q_n\in \mset$, which is possible by Corollary~\ref{approx-linear-alt}. 

Now let $Q_k$ be the matrices such that
\begin{equation*}
 \overline{\mset}x\left(\frac {k-1}n\right) = Q_k x\left(\frac {k-1}n\right).
\end{equation*}
and
\begin{equation*}
 \tilde x\left(\frac {k}n\right) = \prod_{i=1}^k \left(I+\frac 1n Q_i\right)x_0.
\end{equation*}
By \eqref{dist_lin_exp}, 
\begin{equation}\label{dist_to_exp}
d_H(\tilde x(1),  e^\mset x_0)\le \varepsilon. 
\end{equation}
 Now let 
\begin{equation*}
 d_k = \max_{i}\left(x_i\left(\frac {k}n\right) - \tilde x_i\left(\frac {k}n\right) \right)
\end{equation*}
where $x_i$ denotes the $i$th component of the corresponding vector. 

We have 
\begin{align*}
 x\left(\frac {k}n\right) - \tilde x\left(\frac {k}n\right) & \le x\left(\frac {k-1}n\right) + \frac 1n Q_k x\left(\frac {k-1}n\right) + \frac 1n \varepsilon \\
 & \qquad - \tilde x\left(\frac {k-1}n\right) - \frac 1n Q_k \tilde x\left(\frac {k-1}n\right) \\
 & \le d_{k-1} + \frac 1n \| Q_k \| d_{k-1} + \frac 1n \varepsilon \\
 & \le d_{k-1}\left(1+\frac Mn\right) + \frac 1n \varepsilon. 
\end{align*}
In the above equations all constants are used as constant vectors. The values of $d_k$ are therefore bounded from above by the solutions of the recurrence equations $d_{k} = d_{k-1}\left(1+\frac Mn\right) + \frac 1n \varepsilon$, whose solution is $d_k = \frac{\varepsilon}{M}\left(\left(1+\frac{M}{n}\right)^k -1 \right)$ and therefore $d_n\le \frac{\varepsilon}{M}(e^M-1)$. By \eqref{dist_to_exp}, the distance between $x(1)$ and some element in $e^\mset x_0$ is smaller than $d_n+\varepsilon$, which can be made arbitrarily small by taking $\varepsilon$ small enough. Therefore we conclude that $x(1)-\overline {e^\mset}x_0\le 0$ which is equivalent to saying that $x(1)\le \overline{e^\mset}x_0$. This concludes the proof.
\end{proof}

\section{The relation with continuous time imprecise Markov chains}\label{s-rcimc}
The method described in previous sections has a direct applicability in the theory of continuous time Markov chains. Consider Kolmogorov's backward equation 
\begin{equation*}
	\frac{d}{dt} P(t) = QP(t)
\end{equation*}
and let $x_0$ be an arbitrary column vector. Then we denote $x(t) = P(t) x(0)$, where $x(0)=x_0$. The Kolmogorov's equation then translates into 
 \begin{equation*}
	\frac{d}{dt} x(t) = Qx(t),
\end{equation*}
which is exactly the same equation as \eqref{mde}. When imprecision is involved it can be represented through a convex set of generator matrices $\mathcal Q$, which then allows the generalisation of the above equation as in \eqref{initial_conds}-\eqref{diff_ineq_inf}. The solutions of the generalised Kolmogorov equations are then of the form of function $t\mapsto \mathcal X(t)$, where all $\mathcal X(t)$ are interval vectors. Hence we can write $\mathcal X(t) = [\underline x(t), \overline x(t)]$. This induces the operators
\begin{align*}
\underline T(t) & \colon x \mapsto \underline x(t)
\intertext{and}
\overline T(t) & \colon x \mapsto \overline x(t), 
\end{align*}
which are known from the theory of discrete time imprecise Markov chains as the lower and upper transition operator, now dependent on time. Therefore we have a natural correspondence between the solutions of the generalised Kolmogorov's equation and transition operators, similarly as in the precise case of continuous time Markov chains. 

Let us conclude with an example.  
\begin{ex}
Let an imprecise generator matrix be given in terms of a matrix interval $\mathcal Q = [\underline Q, \overline Q]$, where 
\[ \underline Q = \begin{pmatrix}
-7 & 4 & 0 \\ 
2 & -4 & 1 \\ 
0 & 3 & -6 \\ 
\end{pmatrix} \qquad \text{and} \qquad \overline Q = \begin{pmatrix}
-5 & 5 & 2 \\ 
3 & -3 & 2 \\ 
1 & 4 & -4 \\ 
\end{pmatrix}, \] 
and all $Q\in \mathcal Q$ are assumed to have zero row sums. 

Let us approximate $e^{0.2\mathcal Q}$ with $\left(I+\frac {0.2}n \mathcal Q\right)^n$. When $n=80$, the maximal theoretical error is 0.12 and the approximate lower and upper bounds are 
 \[ \underline P(0.2) = \begin{pmatrix}
 0.3164 & 0.3839   & 0.0421       \\ 
0.1545   & 0.5826   & 0.0927       \\
0.0635   & 0.3340   & 0.4019       \\
\end{pmatrix} \quad \overline P(0.2) = \begin{pmatrix}
0.4945   & 0.4984   & 0.2338   \\ 
0.2864   & 0.6921   & 0.2338   \\ 
0.1853   & 0.4432   & 0.5323  
\end{pmatrix} \] 
With $n=200$, the maximal theoretical error drops to 0.02 and the lower and upper bound are approximated with
 \[ \underline P(0.2) = \begin{pmatrix}
0.3181   & 0.3830   & 0.0420       \\ 
0.1541   & 0.5836   & 0.0924       \\ 
0.0633   & 0.3332   & 0.4033       \\ 
\end{pmatrix} \quad \overline P(0.2) = \begin{pmatrix}
0.4957   & 0.4972   & 0.2333   \\ 
0.2858   & 0.6928   & 0.2333   \\ 
0.1849   & 0.4421   & 0.5334  \\ 
\end{pmatrix} \] 
\end{ex}
The presented method is apparently not very efficient in terms of convergence. Therefore, it remains a challenge for further work to find more efficient methods for estimating imprecise continuous time Markov chains. 
\bibliography{refer}

%%%%%%%%%%%%%% Appendix %%%%%%%%%%%%%%%%%%%%%
\appendix 
\section{Inequalities}\label{a-i}
Here we list some of the technical results that are used in the proofs of the main results in the paper. 

Let $f_1, \ldots, f_n$ and $f_1', \ldots, f_n'$ be continuous mappings in a normed space $X$. For every $1\le i \le n$ denote
\begin{equation}\label{bigF}
  F_i = f_1\circ \dots \circ f_i \circ f_{i+1}' \circ \dots \circ f_n'.  
\end{equation}
\begin{lemma}\label{dist_kompozitum}
Suppose that $\|f_i \|, \|f'_i \| \le M$ and $\| f_i - f_i'\| \le d$ for every $1\le i \le n$.  Then 
\[ \| F_n - F_0 \| \le ndM^{n-1}. \]
\end{lemma}
\begin{proof}
 We first notice that 
 \begin{align} 
    \notag \| F_i - F_{i-1} \| & = \| f_1\circ \dots \circ f_i \circ f_{i+1}'\circ \dots \circ f_n'  - f_1\circ \dots \circ f_{i-1} \circ f_{i}'\circ \dots \circ f_n'  \| \notag \\  
    & \le \| f_1\circ \dots \circ f_{i-1} \| \| f_{i}'-f_{i}\|\| f_{i+1}'\circ \dots \circ f_n' \| \label{l4_eq1} \\    
\notag & \le M^{n-1} d . 
 \end{align}
 Further we have
 \begin{align}
    \| F_n - F_0 \| & = \left\| \sum_{i=1}^n F_i - F_{i-1} \right\| \le  \sum_{i=1}^n\left\|F_i - F_{i-1} \right\| \le n M^{n-1}d.\label{l4_eq2}
 \end{align} 
\end{proof}
 The following corollary follows immediately from the above proof.
 \begin{cor}\label{dist_kompozitum_alt}
  Let, under the assumptions of Lemma~\ref{dist_kompozitum}, the inequality
  \begin{equation*}
   \| f_1\circ \cdots \circ f_{i-1}\| \| f'_{i+1}\circ \cdots \circ f'_{n}\| \le A 
  \end{equation*}
  hold for every $1\le i \le n$. 
  Then
  \begin{equation*}
   \| F_n - F_0 \| \le ndA.
  \end{equation*}
 \end{cor}
 In the sequel, we use for matrices an operator norm induced by any of the equaivalent vector norms for finite dimensional vector spaces.  
 \begin{cor}\label{dist_matrix_power}
  Let $A_1$ and $A_2$ be square matrices and let $\| A_i\| \le M$ for $i=1, 2$. Then 
  \begin{equation*}
   \| A_1^n - A_2^n \| \le n M^{n-1} \| A_1-A_2 \|
  \end{equation*}
  for every $n\in\NN$.
 \end{cor}

 \begin{lemma} 
  For every $x\ge 0$ we have the inequality 
  \begin{equation*}
   e^x - 1 -x \le \frac 12 x^2 e^x.
  \end{equation*}

 \end{lemma}
 \begin{proof}
  We have 
  \begin{equation*}
   e^x - 1 - x = \sum_{n=2}^\infty \frac{x^n}{n!} = \sum_{n=0}^\infty \frac{x^2}{(n+1)(n+2)} \frac{x^n}{n!} \le \frac 12 x^2 e^x .
  \end{equation*}
 \end{proof}
 
 \begin{lemma}\label{norm_exp}
  Let $Q$ be an arbitrary matrix. Then 
  \begin{equation*}
   \| e^Q \| \le e^{\|Q \|}.
  \end{equation*}
 \end{lemma}
 \begin{proof}
 \begin{equation*}
  \| e^Q \|  = \left\| \sum_{n=0}^\infty \frac{Q^n}{n!} \right\| \le  \sum_{n=0}^\infty \frac{\|Q\|^n}{n!} = e^{\|Q \|}.
 \end{equation*}  
 \end{proof}
  
 \begin{lemma}\label{linear_one_matrix}
  For every matrix $Q$ we have that 
  \begin{equation*}
   \| e^Q - I - Q \| \le \frac 12 \|Q\|^2e^{\|Q\|}.
  \end{equation*}
 \end{lemma}
 \begin{proof}
  We have that 
  \begin{align*}
   \| e^Q - I - Q \| & = \left\| \sum_{n=2}^\infty \frac {Q^n}{n!} \right\| 
    \le \sum_{n=0}^\infty \left\| \frac{Q^2}{(n+1)(n+2)}\frac {Q^n}{n!} \right\| 
    \le \frac 12 \| Q\|^2 e^{\| Q \|} .
  \end{align*}
 \end{proof}

 \begin{lemma}\label{linear_two_matrices}
  Let $Q_1$ and $Q_2$ be arbitrary matrices such that $\| Q_1\|, \| Q_2 \|\le M$. Then 
  \begin{equation*}
    \| e^{Q_1}e^{Q_2} - I - Q_1 - Q_2 \| \le  M^2 e^M(1.5+e^M).
  \end{equation*} 
 \end{lemma}
 \begin{proof}
  Using the previous lemmas, we obtain
  \begin{align*}
   \left\| e^{Q_1}e^{Q_2} - I - Q_1 - Q_2 \right\| & = \left\| \sum_{m+n>1}\frac {Q_1^mQ_2^n}{m!n!} \right\| \\
   & \le \sum_{n=2}^\infty \left\| \frac{Q_2^n}{n!} \right\| + \sum_{n=1}^\infty \left\| \frac{Q_1Q_2^n}{n!} \right\| + \sum_{m=2}^\infty \left\| \frac{Q_1^m}{m!}e^{Q_2} \right\| \\
   & \le \frac{\| Q_2\|^2}{2} \sum_{n=0}^\infty \left\| \frac{Q_2^n}{(n+2)!} \right\| \\
   & \qquad + \| Q_1 \| \|Q_2\| \sum_{n=0}^\infty \left\| \frac{Q_2^n}{(n+1)!} \right\| \\
   & \qquad + \| Q_1\|^2e^{\|Q_2\|} \sum_{m=0}^\infty \left\| \frac{Q_1^m}{(m+2)!} \right\| \\
   & \le \frac{\| Q_2\|^2}{2} e^{\|Q_2\|} + \| Q_1 \| \|Q_2\| e^{\|Q_2\|} \\
   & \qquad + \| Q_1\|^2e^{\|Q_2\|} e^{\|Q_1\|} \\
   & \le M^2 e^M(1.5+e^M).
  \end{align*}
 \end{proof}
\begin{cor}\label{distance_two}
 Let $Q_1$ and $Q_2$ be arbitrary matrices such that $\| Q_1\|, \| Q_2 \|\le M$. Then
 \begin{equation}
  \| e^{Q_1}e^{Q_2} - e^{Q_1+Q_2} \| \le M^2e^M (1.5+3e^M). \label{c8_eq}
 \end{equation}
\end{cor}
\begin{proof}
It follows from Lemma~\ref{linear_one_matrix} that 
\begin{align*}
\|e^{Q_1+Q_2} - I - Q_1 - Q_2 \| \le \frac 12 \| Q_1 + Q_2 \|^2 e^{\| Q_1+Q_2\|} 
\le 2M^2 e^{2M}, 
\end{align*} 
which together with Lemma~\ref{linear_two_matrices} implies \eqref{c8_eq}. 
\end{proof}
\begin{cor}\label{approx-linear-alt}
 Let $Q_1, \ldots, Q_n$ be matrices with $\| Q_i \| \le M$ for every $1\le i \le n$. Then 
 \begin{equation}\label{approx-linear-alt-eq}
  \| e^{Q_1}\cdots e^{Q_n} - (I+Q_1)\cdots (I+Q_n) \| \le \frac{n}{2}M^2\exp\left(\sum_{i=1}^n\| Q_i\|\right).
 \end{equation}
\end{cor}
\begin{proof}
We follow a similar procedure as in the proof of Lemma~\ref{dist_kompozitum}. Let $f_i = e^{Q_i}$ and $f'_i = (1-Q_i)$ for every $1\le i\le n$, and define $F_i$ as in \eqref{bigF}. By Lemmas~\ref{norm_exp} and \ref{linear_one_matrix}, Eq. \eqref{l4_eq1} and noting that $\| I+Q\| \le e^{\|Q \|}$ we obtain  
 \begin{align*}
  \| F_i - F_{i-1} \| &\le \| e^{Q_1}e^{Q_2}\cdots e^{Q_{i-1}} \| \|e^{Q_i}- I-Q_i\| \| (I+Q_{i+1})\cdots (I+Q_n) \| \\
  & \le \frac 12 M^2 \exp\left( \sum_{i=1}^n{\| Q_i\|} \right), 
 \end{align*}
 which by \eqref{l4_eq2} implies \eqref{approx-linear-alt-eq}. 
\end{proof}

\begin{lemma}\label{refining}
 Let $Q'_i$ and $Q''_i$, for $1\le i \le n$, be matrices such that $\| Q'_i\|, \| Q''_i \| \le M$ for a constant $M$; and let $t'_i$ and $t''_i$ be positive real numbers such that $\displaystyle \sum_{i=1}^{n} t'_i + t''_i = 1$ and $t'_i, t''_i \le q$ for every $1\le i \le n$ and some constant $q$. 
 
 Further, let $t_i = t'_i + t''_i$ and $\displaystyle Q_i = \frac{t'_i}{t_i}Q'_i + \frac{t''_i}{t_i}Q''_i$ for every $1\le i \le n$. 
 
 Then 
 \begin{equation*}
  \left\| \prod_{i=1}^{n} e^{t'_iQ'_i}e^{t''_iQ''_i} - \prod_{i=1}^{n} e^{t_iQ_i} \right\| \le n (qM)^2e^{(q+1)M}\left( 1.5 + 3 e^{qM}\right).
 \end{equation*}

\end{lemma}
\begin{proof}
We again follow the procedure used in the proof of Lemma~\ref{dist_kompozitum}. Denote $f_i = e^{t'_iQ'_i}e^{t''_iQ''_i}$ and $f'_i = e^{t_iQ_i}$ for every $1\le i \le n$. Using Lemma~\ref{norm_exp} we obtain 
\begin{align*}
 \| f_1\ldots f_{k-1} \| & \le \prod_{i=1}^{k-1} \| e^{t'_iQ'_i} \| \| e^{t''_iQ''_i} \| \le \prod_{i=1}^{k-1} e^{t'_iM} e^{t''_iM} = \exp\left\{{M\sum_{i=1}^{k-1}t'_i + t''_i}\right\}
\end{align*}
and similarly, 
\begin{equation*}
 \| f'_{k+1}\ldots f'_{n} \| \le \exp\left\{ M\sum_{i=k+1}^{n}t_i \right\}.
\end{equation*}
Therefore, clearly, 
\begin{equation*}
 \| f_1\ldots f_{k-1} \|\| f'_{k+1}\ldots f'_{n} \| \le e^M.
\end{equation*}
By Corollary~\ref{distance_two} and because of $\| t'_iQ'_i \|, \| t''_iQ''_i \| \le qM$, we have
\begin{equation*}
 \| f_i - f'_i \| \le (qM)^2 e^{qM}(1.5+3e^{qM}). 
\end{equation*}
Finally, by Corollary~\ref{dist_kompozitum_alt}, we obtain
\begin{equation*}
\| F_{n} - F_0 \| \le  n (qM)^2 e^{(q+1)M}(1.5+3e^{qM}), 
\end{equation*}
where notation \eqref{bigF} is used. 
\end{proof}

\end{document}